\documentclass[12pt]{article}
\usepackage{verbatim}
\usepackage{float}
\usepackage{amssymb,amsmath}
\usepackage{amsthm}
\usepackage{graphicx}
\usepackage{pinlabel}
\newtheorem{corollary}{Corollary}[section]
\newtheorem{lemma}[corollary]{Lemma}
\newtheorem{proposition}[corollary]{Proposition}
\newtheorem{theorem}[corollary]{Theorem}
\newcommand{\Prob} {{\bf P}}
\newcommand{\Z}{{\mathbb Z}}
\newcommand{\E}{{\bf E}}

\newcommand{\R}{{\mathbb{R}}}

\newcommand{\D} {{\mathcal{D}}}

\newcommand{\dist}{{\rm dist}}

\def\H{\mathbb{H}}
\def \Im {{\rm Im}}
\def \Re {{\rm Re}}
\def \Arg{{\rm Arg}}
\mathchardef\mhyphen="2D

\begin{document}

\title{Hausdorff measure of SLE curves}

%\author{Gregory F. Lawler \thanks{Research supported by National
%Science Foundation grant DMS-0734151.}\\
%University of Chicago  \\ \\ \\
\author{Mohammad A. Rezaei
\\ Michigan State University}

\maketitle

\begin{abstract}
In this paper we prove that the Hausdorff $d$-measure of $SLE_{\kappa}$ is zero when $d=1+\frac{\kappa}{8}$ for $\kappa<8$.\\
\end{abstract}

\section{Introduction}

 A number of measures on paths or clusters on two-dimensional
lattices arising from critical statistical mechanical
models are believed to exhibit some kind of conformal
invariance in the
scaling limit.
Schramm introduced a one-parameter family of such processes,
now called the
{\em (chordal) Schramm-Loewner evolution with parameter $\kappa$ ($SLE_\kappa$)},
and showed that these give the only possible limits for conformally
invariant processes in simply connected domains satisfying a certain
``domain Markov property''.  He defined the process as a measure on curves
from $0$ to $\infty$ in $\mathbb{H}$, and then used conformal invariance
to define the process in other simply connected domains.\\

Since then, the geometric properties of SLE have been studied extensively. For $\kappa>8$ the curve is space filling so Hausdorff measure is trivial. So from now on we only consider $\kappa<8$. In \cite{Bf}, Beffara proved that the Hausdorff dimension of the $SLE_\kappa$ path is $d=1+\frac{\kappa}{8}$. His method is based on a certain two-point estimate. Following Beffara's work, several other methods used to strengthen and improve his result. Notably in \cite{LW} the authors generalized his two-point estimate to define 2-point Green's function for SLE. Also in \cite{MW} the authors use the coupling between SLE and GFF (Gaussian free field)  to find the Hausdorff dimension of several interesting subsets of SLE curve. However these methods cannot find the Hausdorff measure of the desired set. In \cite{LS}, Lawler and Sheffield defined a $d\mhyphen$dimensional measure of SLE for some values of $\kappa$ which they called \textit{natural parametrization}. Based on its properties, they conjectured the Hausdorff $d\mhyphen$measure of SLE is zero. In fact having zero measure at the dimension is a property that is shared by other random fractals like Brownian motion and fractional Brownian motion to name a few. The underlying reason is the fact that we can cover these random fractals with balls which cover a relatively large proportion of their path.  The proof in this paper uses the same idea as well. Later, in \cite{LZ} construction of natural parametrization was extended to all $\kappa<8$. Finally in \cite{LR}, the authors constructed the natural length as Minkowski content of SLE curves and studied its basic properties. In this paper, we analyze geometry of SLE curves under natural parametrization rather than capacity parametrization to show Lawler and Sheffield conjecture.\\

Here we define Hausdorff measure $\mathcal{H}^d$ so we can state our main theorem and we present SLE theory in the next section. There are a number of places where we can find the definition of Hausdorff measure, for example \cite{Law1}.\\

If $V \subset \mathbb{R}^n$ and $\alpha,\epsilon>0$, let
\[ \mathcal{H}_\epsilon^\alpha(V)=\inf\sum_{n=1}^\infty[{\rm diam}(U_n)]^\alpha,\]
where the inf is taken over all countable collection of sets $U_1,U_2,...$ with $V\subset\bigcup U_n$ and $\rm{diam}(U_n)<\epsilon$. The $\mathit{Hausdorff}\alpha\mhyphen\mathit{measure}$ is defined by
\[\mathcal{H}^\alpha(V)=\lim_{\epsilon \rightarrow 0+} \mathcal{H}_\epsilon^\alpha(V). \]

Since $ \mathcal{H}_\epsilon^\alpha(V)$ is decreasing in $\epsilon$, the limit exists with infinity as a possible value. Note that $\mathcal{H}^\alpha$ is an outer measure. Also it is easy to check that if $\mathcal{H}^\alpha(V)<\infty$, then $\mathcal{H}^\beta(V)=0$ for $\beta>\alpha$, and if $\mathcal{H}^\alpha(V)>0$, then $\mathcal{H}^\beta(V)=\infty$ for $\beta<\alpha$. The Hausdorff dimension of $V$ is defined by
\[
\dim_h(V)=\inf\{\alpha:\mathcal{H}^\alpha(V)=0\}=\sup\{\alpha:\mathcal{H}^\alpha(V)=\infty\}.
\]

By Beffara's result in \cite{Bf} we get $\mathcal{H}^{d-\epsilon}(\gamma)=\infty$ and $\mathcal{H}^{d+\epsilon}(\gamma)=0$. With these notation, the main theorem that we prove is the following.

\begin{theorem} \label{mainresult}
Suppose $\gamma$ is an SLE$_\kappa$ curve in $\mathbb{H}$ from 0 to $\infty$ and $d=1+\frac{\kappa}{8}$. Then we have

\[ \mathcal{H}^d(\gamma)=0. \]
\end{theorem}

In general if $h: [0,\infty) \rightarrow [0,\infty)$ is an increasing function such that $h(0)=0$ we can define (using the same notation as above),
\[ \mathcal{H}_\epsilon^h(V)=\inf\sum_{n=1}^\infty h[{\rm diam}(U_n)],\]
and 
\[\mathcal{H}^h(V)=\lim_{\epsilon \rightarrow 0+} \mathcal{H}_\epsilon^h(V). \]
Note that $\mathcal{H}^\alpha$ that we have defined before is an especial case of this new definition when $h(x)=x^\alpha$. Now $h$ is called an exact gauge function for $V$ if $\mathcal{H}^h(V)$ is a finite and strictly positive number. Having Theorem \ref{mainresult}, we want to raise the following question.\\

{\bf Question:} What is the exact gauge function for SLE curves?\\

For example for Brownian motion in dimension $n \geq 3$, it is well known that the exact gauge function is 
\[
h(x)=x^2\log\log \frac{1}{x}. 
\]
We expect to have a $log$ term correction for SLE as well but in order to answer this probably we need to have higher moments of natural parametrization which is related to the main conjecture of \cite{LR} about Minkowski content of SLE. We hope the methods develop here will be useful to address this finer question as well.\\

%Also, in \cite{LS}, it is conjectured that the Minkowski content of SLE curves exists and equals to the natural parametrization. In this paper we prove the conformal version of this.

%\newpage

%\begin{theorem} \label{Minkowski}
%Take SLE$_\kappa$ curve $\gamma$ in a domain $D$. Put

%\[\mathcal{N}_{t,\epsilon}^*= \{ z \in D: \Upsilon_t(z) \leq \epsilon\},\]

%where $\Upsilon_t(z)$ is conformal radius at $z$ at time $t$. Then for any fixed time $T$

%\begin{equation} \label{eqMinkowski}
%\lim_{\epsilon \rightarrow 0} \E \left[ | \epsilon^{d-2}\; \rm{area}\;(\mathcal{N}_{T,\epsilon}^*)-\Theta_T(D)| \right]=0,
%\end{equation}

%where $\Theta_t(D)$ is natural parametrization in $D$.
%\end{theorem}

\maketitle

\section{SLE and natural parametrization}

\subsection{Schramm-Loewner evolution (SLE) and the Green's function}  \label{SLEdef}

In this section we review the definition of the chordal Schramm-Loewner evolution and introduce our notation. See \cite{Law1} for more details.\\
Consider the Loewner equation \begin{equation}  \label{loew}
\dot g_t(z) = \frac{a}{g_t(z) - V_t} , \;\;\;\;
g_0(z) = z,
\end{equation}
where $V_t$ is a continuous function. As is standard by now, when we set $V_t$ equal to Brownian motion, then we obtain SLE$_\kappa$ curve where $\kappa=\frac{2}{a}$. Throughout this paper, we consider $\gamma:(0,\infty) \rightarrow \mathbb{H}$ as an chordal SLE$_\kappa$ curve in the upper-half plane from $0$ to $\infty$ unless otherwise specified. We will use the scaling property of $SLE$ that we recall in a proposition.

\begin{proposition}  Suppose $U_t$ is a standard
Brownian motion and let $g_t$ be the solution
to the Loewner equation \eqref{loew} with $V_t = U_t$
producing the $SLE_\kappa$ curve $\gamma(t)$.
Let $r > 0$ and define
\[   \hat \gamma(t) = r^{-1} \, \gamma(r^2 t), \;\;\;\;
    \hat g_t(z) = r^{-1} \, g_{r^2t}(rz), \;\;\;\;
       \hat U_t = r^{-1} \, U_{r^2t}. \]
Then $\hat \gamma(t)$ has the same distribution as the $SLE_\kappa$.
\end{proposition}

$SLE_\kappa$ in other simply connected domains is
defined by conformal
invariance.  To be more precise, suppose that $D$ is a simply connected
domain and $w_1,w_2$ are distinct points in $\partial D$.  Let $F: \mathbb{H}
\rightarrow D$ be a conformal transformation of $\mathbb{H}$ onto $D$ with
$F(0) = w_1, F(\infty) = w_2$.  Then the distribution of
\[        \tilde \gamma(t) = F \circ \gamma(t) , \]
is that of $SLE_\kappa$ in $D$ from $w_1$ to $w_2$.  Although
the map $F$ is
not unique, the scaling invariance of $SLE_\kappa$ in $\mathbb{H}$ shows that the distribution is independent of the choice. This measure is often considered
as a measure on paths modulo reparametrization, but we can also consider
it as a measure on parameterized curves.\\

If $\gamma(t)$ is an $SLE_\kappa$ curve with transformations $g_t$
and driving function $U_t$, we write $\gamma_t = \gamma(0,t], \gamma
= \gamma_\infty$, and let $H_t$ be the unbounded component of
$\mathbb{H} \setminus \gamma_t$.  If $z \in \mathbb{H}$ by $T_z$ we mean the first time that $z \not \in H_t$. If $t < T_z$, set
\begin{equation}  \label{may24.1}
Z_t(z) = g_t(z) - U_t=X_t(z)+iY_t(z), \;\;\;\; S_t(z) = \sin \left[\arg Z_t(z)\right], \;\;\;\;
  \Upsilon_t(z) = \frac{Y_t(z)}{|g_t'(z)|}.
  \end{equation}
More generally, if $D$ is a simply connected domain and $z \in D$, we
let $\Upsilon_D(z)$ denote $1/2$  times the conformal radius of
$D$ with respect to $z$, that is, if $F:\mathbb{D} \rightarrow D$ is a
conformal transformation onto $D$ with $F(0) = z$, then $|F'(0)| = 2\Upsilon_D(z)$.
Using the Schwarz lemma and the Koebe $1/4$-theorem, we see that
\begin{equation} \label{Koebe}
 \frac{\Upsilon_D(z)}{2} \leq \dist(z,\partial D) \leq 2 \, \Upsilon_D(z) .
\end{equation}

It is easy to check that if $t < T_z$, then $\Upsilon_t(z)$ as given
in \eqref{may24.1} is the same as $\Upsilon_{H_t}(z)$.  Also,
if $z \not\in \gamma$, then
  $\Upsilon(z) := \Upsilon_{T_z-}(z) = \Upsilon_D(z)$
where $D$ denotes the connected component of $\mathbb{H} \setminus
\gamma$ containing $z$.
Similarly, if $w_1,w_2$ are distinct boundary points on a simply
connected domain $D$
and $z \in D$, we define
\[            S_D(z;w_1,w_2) = \sin[\arg f(z)] , \]
where $f: D \rightarrow \mathbb{H}$ is a conformal transformation with $f(w_1) = 0,
f(w_2) = \infty$.  If $t < T_z$, then
$S_t(z) = S_{H_t}(z;\gamma(t),\infty)$.  If $f:D \rightarrow f(D)$ is
a conformal transformation,
\[     S_D(z;w_1,w_2) = S_{f(D)}(f(z);f(w_1),f(w_2)). \]
If $\partial_1,\partial_2$ denote the two components of $\partial D \setminus
\{w_1,w_2\}$, then 
\begin{equation}  \label{hmeasure}
    S_D(z;w_1,w_2)  \asymp \min\left\{{\rm hm}_{D}(z,\partial_1),
{\rm hm}_D(z,\partial_2) \right\}.
\end{equation}
This bound can be found in \cite{LW}. Here, and throughout this paper, $\rm{hm}$ will denote harmonic measure; that is,
 ${\rm hm}_D(z,K)$ is the probability that a Brownian
motion starting at $z$ exits $D$ at $K$. Also we fix $d=1+\frac{\kappa}{8}$ and $a=\frac{2}{\kappa}$. Let
\begin{equation}  \label{green}
     G(z)
   =|z|^{d-2}  \, \sin^{\frac
  \kappa 8 + \frac 8{\kappa} -2} (\arg z)  = {\rm Im}(z)^{d-2}
    \, \sin^{4a-1}(\arg z)  ,
\end{equation}
denote the {\em (chordal) Green's function for $SLE_\kappa$ (in
$\mathbb{H}$ from $0$ to $\infty$)}.  This function first appeared  in \cite{RS}
and the combination $(d,G)$
can be  characterized up to a multiplicative constant
by the scaling rule $G(rz) =
r^{d-2} \, G(z)$ and the fact that
\begin{equation}  \label{localmart}
    M_t(z) := |g_t'(z)|^{2-d} \, G(Z_t(z))
\end{equation}
is a local martingale.
In general, if $D$ is a simply connected domain with
distinct $w_1,w_2 \in \partial D$, we define
\[    G_D(z;w_1,w_2) = \Upsilon_D(z)^{d-2} \ S_D(z;w_1,w_2)^{4a-1} . \]
The Green's function satisfies the conformal covariance rule
\[   G_D(z;w_1,w_2) = |f'(z)|^{2-d} \, G_{f(D)}(f(z);f(w_1),f(w_2)). \]
Note that if $t < T_z$, then
\[     M_t(z) = G_{H_t}(z;\gamma(t), \infty). \]
The local martingale $M_t(z)$ is not a martingale because
it ``blows up'' at time $t=T_z$.   If we stop it before that time, it
is actually a martingale.  To be precise, suppose that
\begin{equation} \label{tau}
        \tau = \tau_{\epsilon,z} = \inf\{t: \dist\{\gamma_t,z\}
 \leq \epsilon\}
 \end{equation}
 Then for every $\epsilon > 0$, $M_{t \wedge \tau}(z)$ is a
 martingale. % Moreover, on the event $\tau = \infty$, we have
 %$M_\infty(z) = 0$.  Therefore, if $\sigma \leq \tau$ is a stopping
 %time,
%\begin{equation}
%\label{oct18.1}
%     G(z) = \E\left[M_\sigma(z) \right] = \E\left[|g_\sigma'(z)|^{2-d}
%  \, G(Z_\sigma(z))\right].
%  \end{equation}
The following is proved in \cite{LR}, Theorem 2.4.

\begin{proposition} \label{Green2} Suppose $\kappa < 8$, then there are $\alpha, c>0$ depending only on $\kappa$ with the following property. Suppose  $\gamma$ is a chordal $SLE_\kappa$ path from $w_1$ to $w_2$ 
in  a simply connected domain $D$, then for $z \in D$ with $\dist(z,\partial D) \geq 2 \epsilon$, 

\[    \Prob\{\dist(z,\gamma) \leq \epsilon\}
=  c\epsilon^{2-d} G_D(z;w_1,w_2)  
\, [1 + O(\epsilon^\alpha)]. \]
\end{proposition}

%From here we abuse notation a little bit and by $G(z)$ we mean $c_*G(z)$.

\maketitle

\subsection {Two-sided radial SLE} \label{two-sided}

In this subsection we give a quick review of another variation of SLE. For start consider $z \in \H$. We want to define a conformally invariant measure on curves starting from $0$ going to $\infty$ in $H$ which are passing through $z$. This is called $\mathit{two}$-$\mathit{sided \; radial \; SLE_\kappa \; through}\;z$. We can think about it as chordal $SLE_\kappa$ conditioned to go through $z$ and (first are) stops at $T_z$. Because this is an event of probability zero for $\kappa<8$, we define it through a Girsanov transformation.

First we tilt the path by the local martingale $M_t(z)$ that we defined in the previous section. By this we mean that for an event $A$ which is measurable at time $t$ define 
\[
\Prob^*[A]={\bf{E}}[1_AM_t(z)],
\]
where expectation is taken under chordal SLE. By Girsanov theorem, we get that the driving function of the Loewner equation in \eqref{loew} satisfies

\[ dV_t= \frac{(1-4a)X_t(z)}{|Z_t(z)|^2}dt+dW_t, \]
where $W_t$ is a standard Brownian motion in the weighted measure and $X_t,Z_t$ are defined as in (\ref{may24.1}). We should consider the fact that the above equation is valid until $T_z$, the first time that we hit $z$. We have the following proposition which is a modified version of Proposition 2.13 in \cite{LW}. In this version we use Euclidean distance instead of conformal distance. It can be obtained by Theorem 4.2 in \cite{LR}. See \cite{Yao} Proposition 4.2 for a detailed proof.

\begin{proposition} \label{reduce}
%(\cite[proposition 3.12]{LW})
There exist $u>0,\; c<\infty$ such that the following is true. Suppose $\gamma$ is the chordal $SLE_\kappa$ path from 0 to $\infty$ and $z \in \mathbb{H}$. For $\epsilon \leq \rm{Im}(z)$, consider $\tau_\epsilon$ defined in \eqref{tau}. Suppose $\epsilon '<3\epsilon/4$. Let $\mu_1,\mu_2$ be the two probability measures on $\{\gamma(t): 0 \leq t < \tau_\epsilon\}$ corresponding to chordal $SLE_\kappa$ conditioned on the event $\{\tau_{\epsilon'}<\infty\}$ and the two sided radial $SLE_\kappa$ through $z$. Then $\mu_1,\mu_2$ are mutually absolutely continuous with respect to each other and the Radon-Nikodym derivative satisfies
\begin{equation*}
 \left|\frac{d\mu_2}{d\mu_1}-1\right|<c(\epsilon'/\epsilon)^u.
\end{equation*}
\end{proposition}

Whenever we want to consider measure with respect to two-sided radial SLE we use $\Prob^*_z$.

%we need the following lemma to prove Theorem \ref{Minkowski}.

%\begin{proposition} \label{timegreen}
%With the notation as above we have
%\begin{equation}
%\limsup_{\epsilon \rightarrow 0} \epsilon^{d-2} \Prob[\Upsilon_t(z) \leq %\epsilon] \leq G^t(z)
%\end{equation}
%where $G^t(z)=G(z)\Prob^*_z[T_z \leq t]$.
%\end{proposition}

%\begin{proof}

%For this proof put $\rho_n= \inf \{t\;|\; \Upsilon_t(z)<\frac{1}{n}\}$. By Girsanov we have

%\begin{eqnarray*}
%\limsup_{\epsilon \rightarrow 0} \epsilon^{d-2} \Prob[\Upsilon_t(z)\leq \epsilon] & = & \limsup_{n \rightarrow \infty} n^{2-d} \E^*_z[M_{\rho_n}^{-1}(z) 1\{ {\rm hcap}(\gamma[0,\rho_n]<t\}] \\
%& = & \E^*_z[S_{\rho_n}^{1-4a}(z) 1\{ {\rm hcap}(\gamma[0,\rho_n]<t\}.
%\end{eqnarray*}

%We have the second equality by $\Upsilon_{\rho_n}=\frac{1}{n}$ w.p.1.
%By Monotonicity of hcap we have

%\[
%\E^*_z[S_{\rho_n}^{1-4a}(z) 1\{ {\rm hcap}(\gamma[0,\rho_n]<t\}] < \E^*_z[S_{\rho_n}^{1-4a}(z) 1\{ {\rm hcap}(\gamma[0,\rho_{\sqrt{n}}]<t\}].
%\]

%By Lemma 2.2 in \cite{LZ} we have

%\[
%|\E^*_z[S_{\rho_n}^{1-4a}(z) | \mathcal{F}_{\sqrt{n}}]-G(z)|<e^{-\frac{1}{4}\sqrt{n}}.
%\]

%So we get

%\[
%\limsup_{n \rightarrow \infty}\E^*_z[S_{\rho_n}^{1-4a} 1\{ {\rm hcap}(\gamma[0,\rho_{\sqrt{n}}]<t\}] \leq \lim_{n \rightarrow \infty} G(z) \Prob^*_z[ {\rm hcap}(\gamma[0,\rho_{\sqrt{n}}]<t)]= G(z)\Prob^*_z[T_z<t].
%\]

%We have the second equality by continuity of two-sided radial at $z$ which is proved in \cite{Law3}. By this we get what we want.

%\end{proof}

\subsection{Natural parametrization}

As we mentioned above, the natural parametrization defined in \cite{LS} is a candidate for the geometric parametrization of SLE curves which arise in the scaling limits of various discrete models. Its scaling exponent is $d$. Also under some mild conditions, it is the only parametrization that satisfies certain natural conditions (\cite[Proposition 2.2]{LS}). The original definition goes as follows.\\

 Suppose that $D$ is a bounded domain and $z,w$ are
distinct boundary points. Let $\gamma$ denote an $SLE_\kappa$ curve from $z$ to $w$ in $D$ and $t$ is the capacity parametrization inherited from $\mathbb{H}$. Set
\[                \Psi_t(D) = \int_{D_t} G_{D_t}(\zeta;\gamma(t),w)\,  d A(\zeta).
\]

For simplicity, assume that $\Psi_0(D)<\infty$. Because $G_{D_t}(z)$ is a positive local martingale, we get that $\Psi_t(D)$ is a supermartingale. By the Doob-Meyer decomposition, under some conditions, we get a unique increasing process $\Theta_t(D)$ such that $\Psi_t(D)+\Theta_t(D)$ is a martingale. $\Theta_t(D)$ is called $\mathit{the \; natural \; parametrization}$ for $SLE_\kappa$ in $D$.\\

The original definition given in \cite{LS} is only valid in $\mathbb{H}$. They use some cutoff argument instead of assuming $\Psi_0(D)<\infty$ (they define it on a bounded subdomain in $\mathbb{H}$ instead of $\mathbb{H}$). In \cite{LR}, as we mentioned earlier, the authors constructed natural length in terms of Minkowski content and studied basic properties of SLE curves under that. We recall the results from \cite{LR} that we will be needed here.\\

\begin{proposition} \label{natscaling}

If $F:\mathbb{H} \rightarrow D$ is a conformal map, $z,w \in \partial D$, such that $F(0)=z$ and $F(\infty)=w$, then we get
\[
\hat{\Theta}_t(D)= \int_0^t |F'(\gamma(s))|^d d\Theta_s.
\]
where $\hat{\Theta}$ and $\Theta$ are the natural parametrizations in $D$ and $\mathbb{H}$ respectively.
\end{proposition}

This proposition basically shows the $d$-dimensional scaling that we expect from the natural parametrization. The next proposition establishes additivity of the natural length.\\

\begin{proposition} \label{additivity}

Consider $\gamma$ as an $SLE$ curve from $0$ to $\infty$ in $\mathbb{H}$ and $r>0$. Introduce
\[ \gamma^r(t)=\gamma(t+r).\]
Then if $\tau$ is a stopping time and $0<s<t$, we have
\[ \Theta_{t+\tau}-\Theta_{s+\tau}=\Theta_t^\tau-\Theta_s^\tau\]
where $\Theta^\tau$ is the natural parametrization in the domain $H_\tau$.

\end{proposition}

%The next proposition which is in \cite{LZ}, is our main tool to prove Theorem \ref{Minkowski}

%\begin{proposition} \label{Minkhelp}

%Let $\gamma^\epsilon(0,t]=\{ z \in \mathbb{H} \; : \; \sup_{0 \leq s \leq t} M_s(z)\geq \epsilon^{2-d}\}$. Then for any fixed time $T$,

%\begin{equation} \label{eqMinkhelp}
%\lim_{\epsilon \rightarrow 0} \E \left[| \epsilon^{d-2}\; {\rm area}\;(\gamma^\epsilon(0,T])-\Theta_T(D)|\right]=0.
%\end{equation}

%\end{proposition}

\maketitle

\section{Proof of  Theorem \ref{mainresult}}

To prove Theorem \ref{mainresult}, we follow the proof in \cite{MP} for Brownian motion. The strategy is to cover the path of SLE curve with a suitable set of circles which have small $d\mhyphen$measure. In order to do this, following the same method for Brownian motion, we consider the circles that SLE spend an unusually long time in them. We show that there is a good chance that this happens. The necessary estimate \eqref{main1} below which shows this, is more complicated here than in the proof for Brownian motion, due to the fact that SLE is not Markovian.\\

Take $A=[0,1] \times [0,1]$ and $z^0=-\frac{1}{2}+2i$. We prove that $\mathcal{H}^d(\gamma \cap (A+z^0))=0$ with probability 1. The same proof works for $[-m,m] \times [i/m,mi]$ for any $m$, which gives the theorem. Let  $B \subset A$. From now on $B^0$ will mean $B+z^0$. Define
\begin{equation}
\mu(B^0)=\int_0^\infty 1(\tilde{\gamma}(s) \in B^0)ds \;\;\; {\rm for} \; B \subset A,
\end{equation}
where $\tilde{\gamma}(s)$ is the reparametrization of $\gamma(s)$ by natural parametrization and $1(A)$ or $1\{A\}$ is used for indicator function of an event $A$. In other words
\[
\tilde{\gamma}(s)=\gamma(t) \;\; {\rm where} \;\;\; s=\Theta(t).
\]
Note that by this we can write
\[\Theta_\tau(V)=\int_0^\tau 1(\tilde{\gamma}(s) \in V)ds.\]

Fix a big integer $l$ throughout the section which we determine later. For $k>0 \in \Z$, let $\mathcal{S}_k$ be the set of all squares of the form
\[ [n_1l^{-k},(n_1+1)l^{-k}] \times [n_2l^{-k},(n_2+1)l^{-k}], \]
where $n_i \in \{0,1,...,l^k-1\}$. Take $\mathcal{D}_k$ as the set of circumcircles of the squares in $\mathcal{S}_k$. As above, when we write $\mathcal{D}^0_k$ we mean $\mathcal{D}_k+z^0$. Take $\epsilon>0$. Fix $m<M$. We will find them as functions of $\epsilon$ in the proof. We say $D^0 \in \mathcal{D}^0_k$ is $\mathit{big}$ (with respect to $\epsilon$) if
\begin{equation*}
\mu(D^0)>\frac{l^{-dk}}{\epsilon}.
\end{equation*}
Consider the following cover $\mathcal{C}(m,M,\epsilon)=\mathcal{C}$ of $\gamma \cap A^0$.
\begin{itemize}

\item  Maximal big circles $D^0 \in \mathcal{D}^0_k$, $m \leq k <M$, i.e., all big circles that are not contained in another big one for $m \leq k <M$.

\item  Circles in $\D^0_M$ that are not contained in any big circle $D_k$ for $m \leq k <M$ but intersect $\gamma$.

\end{itemize}
This is a cover with sets of diameter at most $\sqrt{2}l^{-m}$. Using this cover we can see that there is $c>0$ such that

\[ c\mathcal{H}^d(\gamma \cap A^0) \leq \left[ \lim_{\epsilon \downarrow 0} \sum_{k=m}^{M-1}\sum_{D^0 \in \mathcal{D}^0_k \cap \mathcal{C}} l^{-dk}+\sum_{ D^0 \in \mathcal{D}^0_M \cap \mathcal{C}} l^{-dM}\right]= \lim_{\epsilon \downarrow 0}[Y_1+Y_2], \]
where
\[ Y_1=\sum_{k=m}^{M-1}\sum_{D^0 \in \mathcal{D}^0_k \cap \mathcal{C}} l^{-dk},
\]
and $Y_2$ is the remainder. For each $\epsilon>0$ we will find $m=m(\epsilon)$ and $M=M(\epsilon)$ such that
\[ \lim_{\epsilon \downarrow 0}\E[Y_1+Y_2]=0. \]
In order to do this we start by estimating $\E[Y_2]$. Put $D^0=D^0_M \subset D^0_{M-1} \subset \cdots \subset D^0_m$ with $D^0_k \in \D^0_k$. For any circle $D$, Let $D^*$ be the circle with same center as $D$ and $100$ times its radius.
Put $\tau_k=\inf\{t>0: \gamma(t) \in D^0_k\}$ and $\tau_k^*=\inf\{t>0: \gamma(t) \in (D^0_k)^*\}$.
We claim that we can take $l$ such that, for every $\epsilon>0$, there is $q=q_\epsilon<1$ such that

\begin{equation} \label{main1}
\Prob\left[\int_{\tau_k^*}^{\tau_{k+1}^*} 1_{D^0_k}(\tilde{\gamma}(s))ds<\frac{l^{-dk}}{\epsilon}\; | \; \tau_{D^0}<\infty, \mathcal{F}_{\tau_k^*}\right]<q,
\end{equation}
for every $m \leq k<M-1$. We will prove this in Lemma 3.1 below. Given this, if   $\tau(D^0)<\infty$, then the probability that $D^0$ is not contained in any big circle is less than $q^{M-m}$, that is
\[ \Prob[D^0 \in \mathcal{C} | \; \tau(D^0)< \infty]< q^{M-m},\]
Combining this with the fact that $\Prob[\tau(D^0)<\infty]\asymp l^{M(d-2)}$ by Proposition \ref{Green2}, we get
\begin{equation} \label{Y_1}
\E[Y_2]<cl^{-Md}l^{2M}q^{M-m}l^{M(d-2)}=cq^{M-m}.
\end{equation}
To estimate $\E[Y_1]$, we have

\begin{equation} \label{Y_2}
Y_1= \sum_{k=m}^{M-1} cl^{-dk} \sum_{D^0 \in \mathcal{C} \cap \D^0_k} 1\{\mu(D^0)> l^{-dk}\frac{1}{\epsilon}\}  <c\epsilon \sum_{k=m}^{M-1}\sum_{D^0 \in \mathcal{C} \cap \D^0_k} \mu(D^0) <5c\epsilon \mu(A^0).
\end{equation}
We have the first inequality because of
\[
l^{-dk}1\left\{\mu(D^0)>l^{-dk}\frac{1}{\epsilon}\right\}<\epsilon \mu(D^0).
\]
Also we have the second inequality because $z$ can not be in 5 different maximal circles. Because $\E[\mu(A^0)]<\infty$, then by \eqref{Y_1} and \eqref{Y_2}, we can find $m$ and $M$ such that
\[ \lim_{\epsilon \downarrow 0}\E[Y_1+Y_2]=0. \]

In other words, we infer that the expected Hausdorff $d-$measure of our cover converges to 0 with a suitable choice of $M=M(\epsilon), m=m(\epsilon)$ as $\epsilon \downarrow 0$. Hence a subsequence converges to zero w.p.1, and as $m$ was arbitrary, we get $\mathcal{H}^d(\gamma \cap A^0)=0$ almost surely. In the rest of the section we prove \eqref{main1} to complete the proof of Theorem \ref{mainresult}. By scaling it is enough to prove the following lemma.

\begin{figure}[H]
		
	\begin{center}
		\labellist
		\small
			\pinlabel $D$ [br] at 300 180
			\pinlabel $D_1$ at 300 220
			\pinlabel $D_1^*$ at 260 260
			\pinlabel $D_2$ at 400 160
			\pinlabel $D_2^*$ at 460 160
			\pinlabel $\gamma_{\tau(D_2^*)}$ at 135 180
			\endlabellist
		%	\centering
		\includegraphics[width=4in]{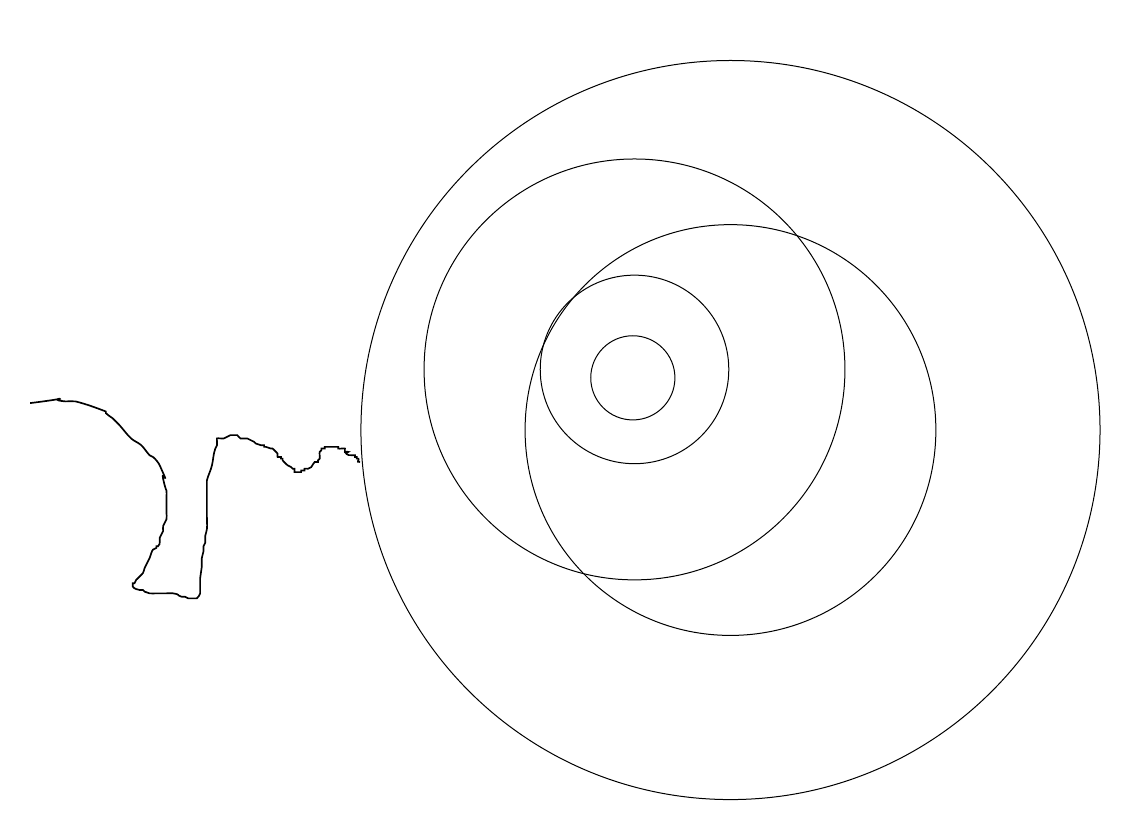}
		\caption{Lemma \ref{circles}}\label{fig1}
	\end{center}
\end{figure}

\begin{lemma} \label{circles}

There exists $l$ large such that the following happens. Suppose $D \subset D_1 \subset D_2$ are three circles in $\H$ with centers at $O,O_1$ and $O_2$ respectively. Suppose the radius of $D_2$ is 1 and the radius of $D_1$ is $1/l$. Also put $D^*_1$ and $D^*_2$ as above i.e they are circles with centers at $O_1$ and $O_2$ and radius $100/l$ and $100$ respectively. Then for every $N>0$ there exists $q>0$ such that

\begin{equation} \label{mainest}
\Prob\left[\Theta_{\tau(D^*_1)}(D_2)>N | \tau_{D}<\infty,\mathcal{F}_{\tau(D^*_2)}\right]>q.
\end{equation}

\end{lemma}

\begin{proof}
Let $\tau'=\inf\{t| \Upsilon_t(O_2)=20\}$. Note that

\begin{equation} \label{distest}
\Upsilon_{\tau(D^*_2)}(O_2)>25 \;\;\;\;\;\; 10<\dist(O_2,\gamma_{\tau'})<40,
\end{equation}
by \eqref{Koebe}. Define the event $E=\{S_{\tau'}(O_2)>1/4\}$. Recall that $\Prob_z^*$ is the measure of two-sided radial SLE through $z$. By Lemma 2.2 in \cite{LZ}, there is $p_1>0$, such that

\[ \Prob_{O_2}^*[E|\mathcal{F}_{\tau(D^*_2)}]>p_1.   \]
Note that until time $\tau',$ $\Prob_{O_2}^*$ and $\Prob_{O}^*$ are mutually absolutely continuous with bounded Radon-Nikodym derivative. Therefore
\[ \Prob_{O}^*[E|\mathcal{F}_{\tau(D^*_2)}]>p'_1,   \]
for some $p'_1$. So by Proposition \ref{reduce}

\[ \Prob[E|\mathcal{F}_{\tau(D^*_2)},\;\tau(D)<\infty]>p_2,\]
for some $p_2>0$. Put $Z_{\tau'}: H_{\tau'} \rightarrow \H$ as in \eqref{may24.1}. If $U=D^*_1 \cup D_2$, then there is a universal $\theta_0>0$ such that if $z \in U$ then $\sin(\Arg(Z_{\tau'}(z)))>\theta_0$ on the event $E$.\\

Note that for $l$ large enough, there is a ball $V$ of radius $1/50$ in $D_2$ which has distance at least $1/3$ to $D^*_1$. Under the event $E$, let $f = bZ_{\tau'}$ where $b>0$ is chosen so that $|f'(O_2)|=1$.  By the distortion theorem ( see \cite[Section 3]{Law1}), for some constant $c>0$, we have

\begin{equation} \label{derest}
\frac{1}{c}<|f'(z)|<c,
\end{equation}
for any $z \in U$. Take $l$ large enough such that $2cl^{-1}<\min\{\frac{1}{40c},\frac{1}{10000}\}$. By Propositions \ref{natscaling} and \ref{additivity}, it is suffices to show

\begin{equation*}
\Prob \left[\Theta_{\tau^*}(f(V))>N|\tau(f(D))<\infty \right]>p>0,
\end{equation*}
where $\tau^*$ is the first time that SLE hits $f(D^*_1)$.\\

By \eqref{derest}, on the event $E$, there is $n$ fixed such that $f(U) \subset R=[-n,n] \times [i/n,ni]$. Also there are fixed $r_1,r_2,\delta>0$ and $z,w \in f(U)$ such that $B_{r_1}(z) \subset f(V),\;f(D^*_1) \subset B_{r_2}(w) \subset R$ and we have $\dist(z,w)=r_1+r_2+\delta$. By taking $L=f(D)$ we need to show the following lemma to finish the proof.

\end{proof}

\begin{lemma}
For every $C,n,\delta>0$, there exists $p>0$ such that the following holds.
If $r_1>\delta$, $r_2>0$, $\dist(z,w)>r_1+r_2+\delta$,
\[B_{r_1}(z),B_{r_2}(w) \subset \left[-n,n\right] \times \left[\frac{i}{n},ni\right],\]
and $L \subset B_{r_2}(w)$, then

\begin{equation} \label{main2}
\Prob[\Theta_{\tau_L}(B_{r_1}(z))>C| \tau_L<\infty]>p.
\end{equation}

\end{lemma}

%At this point the notation is too complicated so we want to just use new ones. Let us review what we need. Because of the \eqref{derest} and existence of $\theta_0$ we need to show the following.

%\begin{equation} \label{reducedmain1}
%\Prob \left[\Theta_{\tau(B_{r_2}(w))}(B_{r_1}(z))|\tau(L)<\infty\right]>p>0,
%\end{equation}

%where $p$ is independent of $z,w$ and $L$.

\begin{proof}

By scaling we assume that $\delta=1$. Also for the moment let us assume $z$ is fixed. Later, a compactness argument will imply that we may choose $p$ independent of $z$. To show \eqref{main2}, we define a geometric event which has positive probability such that the natural length is big on it with positive probability. Note that we can make $r_1$ smaller so we can assume that $r_1=1$. Consider the following path $\iota_1$ from $0$ to a point at distance $1/2$ of $z$. Take

\begin{figure}[H]

\begin{center}
	\labellist	
	\small
	\pinlabel $z$ [br] at 140 155
%	\pinlabel $l_0$ at 540 180
	\pinlabel $l$ at 260 115
%	\pinlabel $\tau_1$ at 600 167
	\pinlabel $h(B)$ at 230 83
	\endlabellist
%	\centering
	\includegraphics[width=4in]{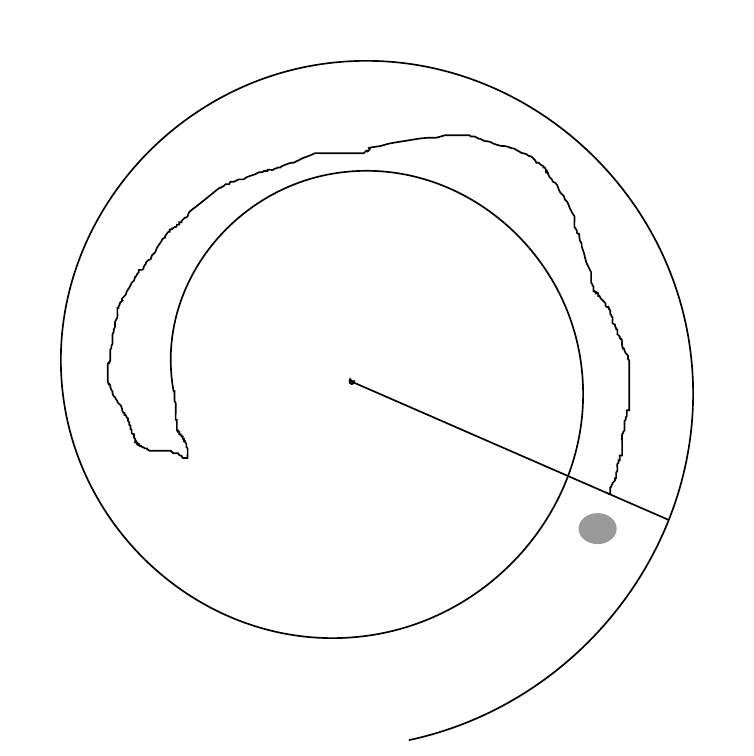}
	\caption{Spiral $\lambda$}\label{fig2}
	\end{center}
\end{figure}

\[
\iota_1=\left[0,\frac{i}{2n}\right] \cup \left[\frac{i}{2n},n+\frac{i}{2n}\right] \cup \left[n+\frac{i}{2n},n+\Im(z)i\right] \cup \left[n+\Im(z)i,z+1\right] \cup \lambda,
\]
where $\lambda$ is a spiral around $z$ as described below (see Fig. \ref{fig2}). %We put the following conditions on $\lambda$. Consider $\lambda:[0,1] \rightarrow B_1(z)$ such that

%\begin{itemize}

%\item $\lambda(0)=z+1$ \;\;\;\;\; $\dist(z,\lambda(1))=1/2$.

%\item It is always wrapping clockwise around $z$.

%\item  Any straight half line from $z$ which we call it a "ray" will hit $\lambda$ in exactly $k$ points.

%\item  If $z+s_1 e^{i\theta}$ and $z+s_2 e^{i\theta}$ are on $\lambda$ then $\frac{1}{32k}<|s_1-s_2|$.

%\end{itemize}
%So it is just an ordinary spiral but the arms are spread out.
Define $\lambda:[0,1] \rightarrow B_1(z)$ by $\lambda(t)=e^{2i\pi kt}(1-\frac{t}{2})+z$. We will determine $k$ based on $C$ later but for now it is a large number satisfying  $\frac{1}{k^2}<\frac{1}{100}$. Let $\iota_2$ be the curve resulting from the reflection of the construction of $\iota_1$ about the $y\mhyphen$axis. We use $\iota_2$ if $\Re(w)>\Re(z)$ and otherwise we use $\iota_1$. Assume  $\Re(w)>\Re(z)$. Since $\dist(z,w)>2$, the construction of $\iota_2$ ensures that 
\begin{equation} \label{iota}
\dist(\iota_2,B_{r_2}(w)) \geq \delta',
\end{equation}
for some $\delta'>0$. Consider the strip $S=\{z \in \mathbb{H}\;|\; \dist(z,\iota_2)<\frac{1}{k^2}\}$.
Consider the event $V$ that SLE curve $\gamma$ stays in $S$ until the time $\tau$ that it reaches distance $1/2+\frac{1}{k^2}$ of $z$. We have
\[ P[V]>0. \]
We will justify this later. By \eqref{iota} we have $\{M_t(w)\; ; t<\tau\}$ is bounded away from zero and $\infty$ if $\gamma$ stays in $S$. Here is the reason. Since $\dist(z,w)>2$, the harmonic measure of $[0,\infty]$ from $w$ is bounded below just by considering the fact that Brownian motion can stay away from $S$ until it reaches $[0,\infty]$. The same argument holds for $[-\infty,0]$ as well. So we get
\[ P^*_w[V]>0. \]
Therefore by Proposition \ref{two-sided}, we get
\[
P[V|\tau(L)<\infty]>0.
\]
\begin{figure}[H]
	
	\begin{center}
		\labellist
		\small
		%	\pinlabel $z$ [br] at 310 180
			\pinlabel $\iota_1$ at 350 45
			\pinlabel $0$ at 271 10
		%	\pinlabel $\tau_1$ at 600 167
		%	\pinlabel $h(B)$ at 430 125
			\endlabellist
		%	\centering
		\includegraphics[width=4in]{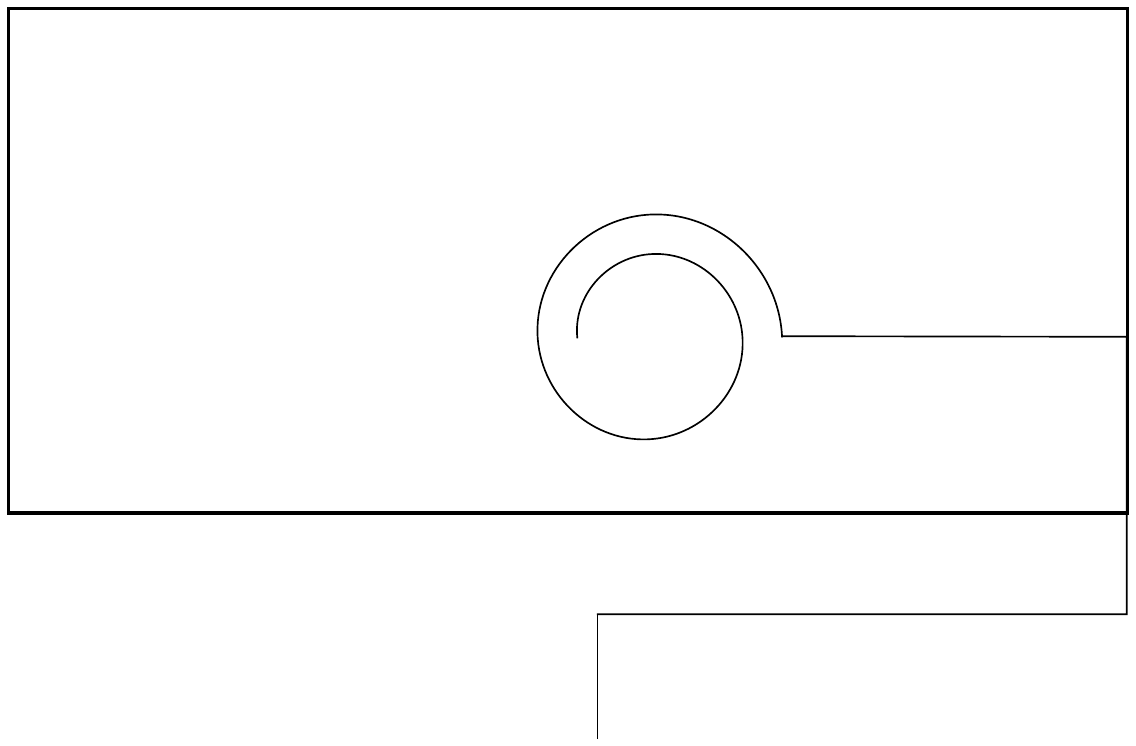}
		\caption{Path $\iota_1$}\label{fig3}
	\end{center}
\end{figure}

Now we want to construct a set $\mathcal{L}$ of crosscuts in the domain $H_\tau$. Generally by crosscut in a domain $D$ we mean a (non-self intersecting) curve from point $A$ to $B$ where $A,B$ are on the boundary of $D$ and the curve between them lies inside $D$. In the present problem, take $z_1 \in H_\tau$ such that $1/2+2/k<\dist(z_1,z)<1-2/k$ and $\dist(z_1,\lambda)>2/k^2$. Write $z_1=z+re^{i\theta}$ where $1/2+2/k<r<1-2/k$. Consider $r_1,r_2$ such that $z_1 \in (z+r_1e^{i\theta},z+r_2e^{i\theta}) \subset H_\tau$ and $z+r_1e^{i\theta},z+r_2e^{i\theta} \in \gamma_\tau$. There is a crosscut in $H_\tau$ which separates $z$ from $\infty$ in $H_\tau$. Here is one way that we can justify this. Take the curve $\tilde{\lambda}$ which is the reverse of $\lambda$. This is a curve starting at $z-1$ ends at $z-1/2$ and has distance at least $1/5k$ to $\lambda$. Also $z$ and $z-1/2$ are on the same connected component of $H_\tau \setminus l$. Now it is easy to see that $\tilde{\lambda}$ will hit $l$ in exactly one point in a transverse way. So if we consider the map $Z_\tau$ it will send $\tilde{\lambda}$ to a curve which hits $Z_\tau(l)$ exactly in one point in a transverse way and can be continued to $\infty$ without hitting $Z_\tau(l)$ again. So by Jordan curve theorem $Z_\tau(z-\frac{1}{2})$ is in the bounded part of $\mathbb{H} \setminus Z_\tau(l)$ which is the claim.\\

\begin{figure}[H]
		
	\begin{center}
		\labellist
		\small
			\pinlabel $z$ [br] at 265 200
		%	\pinlabel $l_0$ at 540 180
			\pinlabel $l$ at 430 193
		%	\pinlabel $\tau_1$ at 600 167
		%	\pinlabel $h(B)$ at 430 125
			\endlabellist
		%	\centering
		\includegraphics[width=4in]{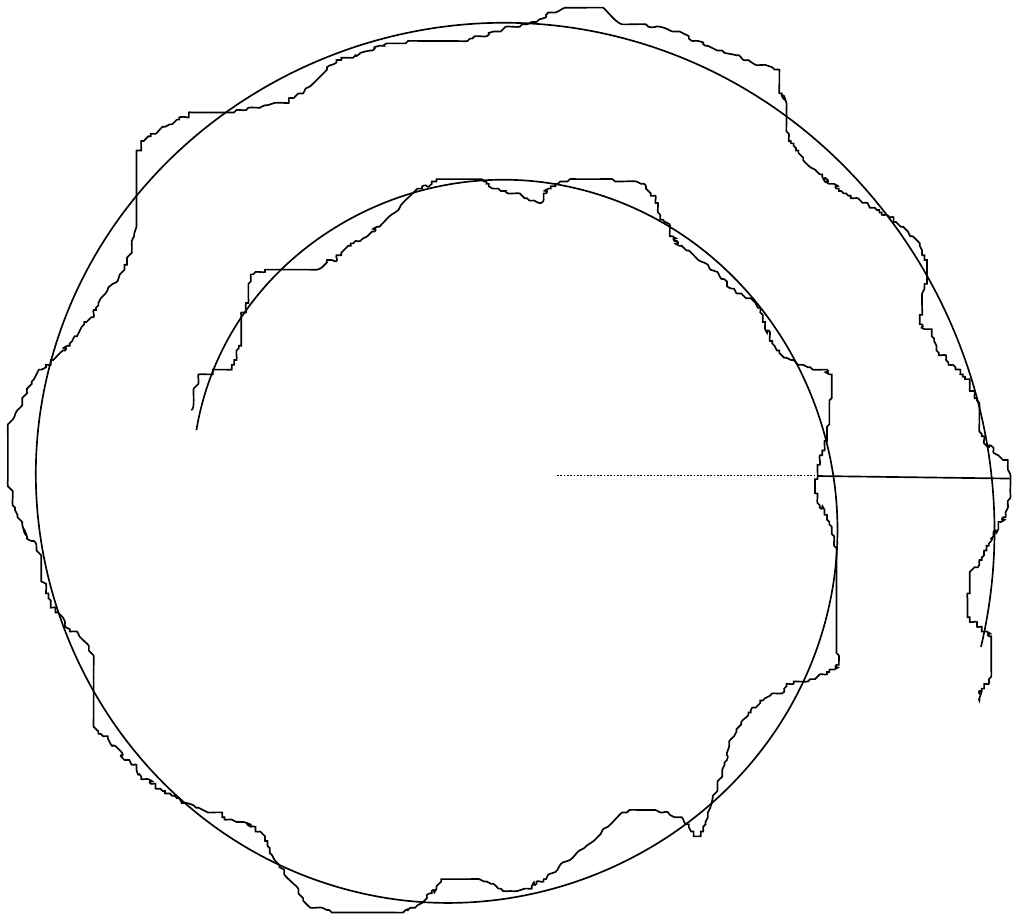}
		\caption{Crosscut $l$}\label{fig4}
	\end{center}
\end{figure}

The reason that we define this family is the following. We want to divide the space in the spiral into regions. Then we show has a positive chance to spend a reasonable time in any of these regions and these events are almost independent. This gives the desired estimate that we are looking for.\\

Take $\mathcal{L}$ be the set of all of possible crosscuts that we get in this way by changing the point $z_1 \in H_\tau$ with the property  $1/2+2/k<\dist(z_1,z)<1-2/k$ that we required above. Note that this is an uncountable set and many $z_1$'s can lie on the same arc so produce the same arc. If $l \in \mathcal{L}$ take $m(l)$ as the midpoint of $l$. We can make $\mathcal{L}$ a metric space by
\[
\dist(l_1,l_2)=\inf\{{\rm length}(\alpha)\;|\; \alpha:[0,1] \rightarrow H_\tau, \alpha(0)=m(l_1), \alpha(1)=m(l_2)\}.
\]
In this metric we can see ${\rm diam}(\mathcal{L}) \geq k/2$. Also if $\dist(l_1,l_2)<1/10$ then $\dist(l_1,l_2) \asymp \dist(m(l_1),m(l_2))$. Also we can make $\mathcal{L}$ an ordered set. We write $l < l'$ if $l'$ separates $l$ from infinity in  $H_\tau \setminus l'$ and $l'<l$ otherwise. To see why this is a well-defined ordered set, consider the map $Z_\tau$. $l$ and $l'$ will go to two crosscuts $l_1=Z_\tau(l)$ and $l'_1=Z_\tau(l')$ in $\mathbb{H}$ that have no intersection. Consider $x_1=l_1 \cap [0,\infty]$ and $y_1=l_1 \cap [-\infty,0]$. Similarly consider $x'_1$ and $y'_1$. Suppose $x_1<x'_1$. Then $l'_1$ is a curve that separates $x_1$ from $\infty$ as well. By this if $l_1$ is not in the bounded part of  $\mathbb{H} \setminus l'_1$ then it has to hit $l'_1$ which is a contradiction since we assumed $l$ and $l'$ have no intersection.\\

%Consider a given ray $l_0$ and take an integer $1<j<k$. Let $l_0^j$ be the part of $l_0$ between arm $j$ and $j+1$ of $\gamma$. It is a crosscut in $\H_\tau$ which separates $z$ from $\infty$.%
Consider $l \in \mathcal{L}$. Since  $\gamma$ goes to $\infty$ w.p.1, there is a first time $\tau_1$ after $\tau$ when it hits $l$. Consider the map $Z_{\tau_1}$ and take $A=f_{\tau_1}(\rm{Re}(z)=0)$. Let $l_1 \in \mathcal{L}$ such that $\dist(l,l_1)=2/k$ and $l<l_1$. Note that $A \cap l_1 \not = \emptyset$ because the imaginary axis is going to $\infty$ in $\mathbb{H}$ so because $l_1$ is a crosscut in $H_{\tau_1}$ we have this. In fact we guess $l_1 \cap A$ has exactly one element but we do not need it. Take $z_0 \in l_1 \cap A$. We claim that $\dist(z_0,\gamma_{\tau_1}) \asymp \frac{1}{k}$. This follows from harmonic measure consideration \eqref{hmeasure} and the Beurling estimate. Put $h^{-1}=bZ_{\tau_1}$ where $b$ is chosen such that $h(z_0)=i/k$. By the Koebe-1/4 theorem we get $|h'(z_0)| \asymp 1$. Now let $B=\{z \in \mathbb{H}| \; \frac{1}{2k}<{\rm Im}(z)<\frac{2}{k}, \frac{\pi}{4}<Arg(z)<\frac{3\pi}{4}\}$. The expected time that SLE spends in $B$, in the natural parametrization,  is $O(k^{-d})$. This is a result of the fact that $\int_B G(z)dA(z)=ck^{-d}$ for some constant $c$. So there is $p>0$ and $c>0$ such that
\[\Prob\{\Theta_\infty(B) \geq ck^{-d} \}>p.\]
$p$ and $c$ are independent of $k$ because of the scale invariance of SLE.\\

Since we have $|h'(z)| \asymp 1$ for all $z \in B$ by the growth theorem,  by Proposition \ref{natscaling} we get (perhaps with a different $c$)
\[\Prob\{\Theta_\infty(h(B)) \geq ck^{-d} \;| \; \mathcal{F}_{\tau_1}\}>p>0,\]
where $p$ and $c$ are independent of $j$ and $k$. Also we get that ${\rm diam}(B)<\frac{c_1}{k}$ for some $c_1$. Now we claim that there is a fixed $N$ large enough, independent of $k$ and $l$ with the following property: Suppose $l_2 \in \mathcal{L}$  with distance $\frac{N}{k}$ of $l$ such that $l<l_2$. Then,
\[\Prob\{\gamma \; {\rm hits} \;\; l_2 \;\; {\rm and \;\; comes\;\; back\;\; and \;\;hits} \;\; B\; |\; \mathcal{F}_{\tau_1}\}<p/2.\]

To justify this, consider $l_3 \in \mathcal{L}$ within distance $\frac{2c_1+2}{k}$ of $l_1$ such that $l<l_3$. Because ${\rm diam}(B)<\frac{c_1}{k}$ if we hit $l_2$ and come back to hit $B$ we have to hit $l_3$. So it is enough to prove
\[\Prob\{\gamma \; {\rm hits} \;\; l_2 \;\; {\rm and \;\; comes\;\; back\;\; and \;\;hits} \;\; l_3\; |\; \mathcal{F}_{\tau_1}\}<p/2.\]

In order to show this we estimate the excursion measure between $l_2$ and $l_3$ then we use \cite{LR} Lemma 2.6 and Lemma 2.7 to conclude the proof. In this case it is easier to estimate extremal distance (see [Section 3.7, \cite{Law1}] for definition). Then the relation between excursion measure and extremal distance gives the needed result. Considering the area $A'$ between $l_2$ and $l_3$ in $H_\tau$ we have
\[
{\rm Area}(A') \leq 10N/k.
\]
We have this because we can put $A'$ in a rectangle of side length $2/k+2/k^2$ and $2N/k$ because $\gamma_\tau \subset S$. Then Lemma 3.74 in \cite{Law1} gives us that the extremal distance between $l_2$ and $l_3$ is at least $cN$ for some $N$ if we assume $k>10N$. This is what we need.\\

In the proof so far we used the measure for SLE in the upper half plane. We should change the measure on $\gamma$ from SLE in the upper half plane to two sided radial SLE towards $w_2$. Notice that these two have the Radon-Nikodym derivative equal to $M_t(w_2)$ by the definition in Section \ref{two-sided}. It is easy to see that  $M_t(w_2)$ for times $t$ after hitting the disk of radius $1/2+\frac{1}{k^2}$ by $\gamma$, is bounded by a universal constant by the conditions that we put on $\lambda$ until we exit $B_{1-2/k}(z)$. Then we change the measure from two sided radial to SLE conditioned to $\tau(L)<\infty$ by Proposition \ref{reduce}.\\

In order to finish the proof of the lemma consider the set of $l_i,\; 1 \leq i \leq \frac{k^2}{100N}$ in $\mathcal{L}$ such that the distance between any two of them is at least $N/k$. Then we can take a set $B_i$ for any of $l_i$ like above. We proved that the probability that we spend $ck^{-d}$ in $B_i$ before hitting the next crosscut is at least $p>0$ independent of the previous events. By this, we get with probability $1/2$, SLE spends at least $ck^{2-d}$ in the spiral for some $c>0$ so by setting $k=\max\{\left(\frac{1}{c\epsilon}\right)^{\frac{1}{2-d}},100N\} $ we get the result.\\

So far we worked with a fixed point $z$ and showed existence of $p$ for that specific point. Note that if we consider another point $z'$ in the ball $B(z)=B_{\frac{1}{100k^2}}(z)$, the proof above goes through if we change $z$ to $z'$ and shows existence of $p'>0$ which works for all $z'$ in $B(z)$. We just cover the domain by disks $B(z)$ and compactness of the domain concludes the proof.

\end{proof}

So it remains just to prove the following lemma.

\begin{lemma}
With the notation as above we have $\Prob[V]>0$.
\end{lemma}
\begin{proof}
In order to show this first we extend $\iota_2$ in the following way. Take line segment $l=[z-1/2,z+1/2]$. Consider a fixed simple curve $\iota_3$ starting at $z-1/2$ in $\mathbb{H}$ which goes to the point $z_3=i9\sqrt{C}$ where $C={\rm hcap}(S)$ and has distance at least $\frac{k}{10}$ to $\iota_2$ and $\iota_3 \cap l=z$ (remember that $S=\{z \in \mathbb{H}\;|\; \dist(z,\iota_2)<\frac{1}{k^2}\}$). We can construct such a curve $\iota_3$ using $\tilde{\lambda}$. Consider
\[\iota_4=\iota_2 \cup l \cup \iota_3,\]
and take $S'=\{z \in \mathbb{H} \;|\; \dist(z,\iota_4)<\frac{1}{k^2}\}$.
Consider the event $V'$ that SLE curve stays in $S'$ until time $2C$ in capacity parameterization. This event has positive probability because if we change the Lowner driving function for $\iota_4$ a little bit we stay in $S'$ by \cite{Law1}[Proposition 4.47] and also ${\rm hcap}(S')>3C$. Also, under this  event, we claim that SLE gets to the ball $B=B_{\frac{1}{k^2}}(z-1/2)$, because otherwise it just stays in $S$, which has capacity $C$. This is a contradiction. So $V' \subset V$ and we are done.
\end{proof}

I want to thank my adviser, Greg Lawler, who introduced the problem and the whole area of SLE to me. Also I want to thank Brent Werness and Dapeng Zhan for many useful conversation about the problem and comments on earlier drafts of it.

\end{document}